\documentclass[reqno,a4paper,11 pt]{amsart}
\usepackage{mathtools}
\usepackage{amsmath}
\usepackage{amssymb}
\usepackage{braket}
\usepackage[numbers]{natbib}
\usepackage[draft]{todonotes}
\usepackage{enumerate}
\usepackage{verbatim}
\usepackage[a4paper,top=2cm,bottom=2cm,left=2cm,right=2cm,bindingoffset=5mm]{geometry}

\newtheorem{theorem}{Theorem}
\newtheorem{proposition}{Proposition}
\newtheorem{lemma}{Lemma}

\newtheorem{definition}{Definition}
\newtheorem{corollary}[theorem]{Corollary}
\newtheorem{remark}{Remark}
\newtheorem{hypothesis}{Hypothesis}

\newenvironment{system}
{\left\lbrace\begin{array}{@{}l@{}}}
{\end{array}\right.}
\DeclarePairedDelimiter{\abs}{\lvert}{\rvert}
\DeclarePairedDelimiter{\norm}{\lVert}{\rVert}

\newcommand{\lsz}{\left\lbrace}
\newcommand{\psz}{\right\rbrace}
\newcommand{\lav}{\left|}                       
\newcommand{\rav}{\right|}                      
\newcommand{\lbr}{\left[}                       
\newcommand{\rbr}{\right]}                      
\newcommand{\lp}{\left(}                        
\newcommand{\rp}{\right)}                       

\newcommand{\R}{\mathbb{R}}

\newcommand{\E}{\mathbb{E}}
\newcommand{\mP}{\mathbb{P}}
\newcommand{\Fcal}{\mathcal{F}}      

\newcommand{\Y}{\mathcal{Y}}
\definecolor{capri}{rgb}{0.0, 0.75, 1.0}

\title[Ergodic SMP]{Ergodic maximum principle for stochastic systems}


\begin{document}

\author{Carlo Orrieri}
\address[C. Orrieri]{Dipartimento di Matematica, Sapienza Universit\`a di Roma. Piazzale Aldo Moro 5, 00185 Roma, Italia}
\email{orrieri@mat.uniroma1.it}

\author{Gianmario Tessitore}
\address[G. Tessitore]{Dipartimento di Matematica e Applicazioni, Universit\`a di Milano-Bicocca. via Cozzi 55, 20125 Milano, Italia}
\email{gianmario.tessitore@unimib.it}

\author{Petr Veverka}
\address[P. Veverka]{Institute of Information Theory and Automation, Czech Academy of Sciences, Pod Vod\'arenskou v\v{e}\v{z}\' i 4, Praha 8, 182 08, Czech Republic}
\email{panveverka@seznam.cz}

\begin{abstract}
We present a version of the stochastic maximum principle  (SMP) for ergodic control problems. In particular we give necessary (and sufficient) conditions for optimality for controlled dissipative systems in finite dimensions. 
The strategy we employ is mainly built on duality techniques. We are able to construct a dual process for all positive times via the analysis of a suitable class of perturbed linearized forward equations. We show that such a process is the unique bounded solution to a Backward SDE on infinite horizon from which we can write a version of the SMP.
\end{abstract}

\maketitle

\section{Introduction}

\noindent We consider an optimal control problem with the following controlled \textit{dissipative stochastic state equation}
\begin{equation}
\begin{system}
dX_t = b(X_t,u_t)dt + \sigma(X_t,u_t)dW_t, \qquad t\geq 0,\\
X_0 = x,
\end{system}
\end{equation}
and an \textit{ergodic cost functional} (e.g. a functional that  depends only on the asymptotic behaviour of the state and of the control) such as:
\begin{align}
J^{\inf} (u(\cdot)) &= \liminf_{T \to \infty} \frac{1}{T}\E\int_0^T f(X_t,u_t)dt, \label{cost inf}\\
J^{\sup} (u(\cdot)) &= \limsup_{T \to \infty} \frac{1}{T}\E\int_0^T f(X_t,u_t)dt. \label{cost sup}
\end{align}
In the above  the state $X$ is a $\R^n$-valued process and $(W_t)_{t\geq 0}$ is a $d$-dimensional Wiener process. Moreover the drift $b$ and diffusion $\sigma$ satisfy a joint monotonicity condition.
Finally the control process $(u_t)$ is progressively measurable and takes values in a non-empty convex subset $U\subset \R^l$.

\noindent We refer to this setting as \textit{ergodic control problem}.  The choice of the functionals refers to "minmin" and "minmax" formulation. Our aim is to find a correct formulation of the stochastic maximum principle (SMP) in the sense of Pontryagin, by means of which we have at our disposal some necessary (and sufficient) condition for optimality.

Alternatively, under stronger regularity assumptions, one can 
use the dynamic programming and derive the Hamilton-Jacobi-Bellman equation whose solution gives the optimal cost and the optimal feedback control. In finite dimensions, the first result was obtained in the paper by Mandl \cite{mandl1964control}, later generalized by Bokar and Gosh in  \cite{borkar1988ergodic}.  For further generalizations of such an HJB approach, both in finite and infinite dimensional framework both by analytic and by probabilistic tools see e.g. \cite{goldys1999ergodic, fuhrman2009ergodic, hu2015probabilistic, debussche2011ergodic, richou2009ergodic, bensoussan1987equations, arisawa1998ergodic}
We also refer to \cite{goldys2011stochastic} for a survey on recent results obtained in this direction.

Nevertheless, it is by now well known that, even if it only provides necessary ( only under strong convexity requirements also sufficient) optimality conditions, the SMP normally requires much less regularity and structural condition allowing for instance to easily include the case of control dependent diffusion. The first general formulation of the SMP for finite horizon controlled stochastic systems in finite dimensions was obtained by Peng in \cite{peng1990general}. After this seminal paper, many directions have been followed by many authors. For what concerns ergodic costs, though, the theory is not yet fully developed. As far as we know, the only version of necessary and sufficient condition for optimality goes back to the paper by Kushner \cite{kushner1978optimality} in 1978, in which no backward stochastic equation appeared.
In that framework the author adopted a martingale solution approach and considered only Markov feedback controls. The system is also assumed to be stable for each control. Under these assumptions, for each stationary Markov control there exists a unique invariant measure $\mu_u(\cdot)$ such that the initial cost functional can be rewritten in the following way
\begin{equation*}
\lim_{T \to \infty} \E_{x_0}^u \frac{1}{T} \int_0^T f(X_t,u(X_t)) dt = \int f(x,u(x)) \mu_u(dx).
\end{equation*}
Using this formulation, Kushner derived a necessary and sufficient condition for $u(\cdot)$ to be optimal, which he called a ``dynamic programming like'' condition.
Let us also mention a recent preprint \cite{cohen2015classical} in which the authors give some sufficient condition for optimality, studying the adjoint Backward SDE, as well as Feller property and exponential ergodicity of the controlled process. As in the present paper the adjoint BSDE is  multidimensional in an  infinite horizon.
The point is that the approach chosen in \cite{cohen2015classical} to prove well posedness of such an equation relies on Girsanov argument and seems to work  under commutativity requirements that are satisfied when $n=1$ or when $\sigma$ is constant. Also see \cite{guatteri2009ergodic} for infinite horizon multidimensional BSDEs in the context of linear quadratic stationary optimal control.
 
Our formulation is fairly general. We do not impose the existence of a limit in the formulation of the cost functional and we consider general progressive controls. Moreover, notice that the convexity assumption on the control actions is a natural choice for the ergodic control problems. Indeed, due to the dissipativity of the system, a spike variation argument is not sufficient to extract useful information on the behaviour of the system at infinity.  In the present paper we deduce a version of the maximum principle written in terms of the unique bounded solution to a multidimensional backward SDE on infinite horizon
\begin{equation}
-dp_t = \left[ D_xb(X_t,u_t)^*p_t + D_x\sigma(X_t,u_t)^*q_t - D_xf(X_t,u_t)\right]dt - q_tdW_t.
\end{equation}
As far as we know, a well-posedness result for backward equations  of this form is new. The major difficulty to overcome is the lack of integrability in time of the forcing term of the equation. Due to the hypothesis on the state equation we can guarantee that
\[\sup_{t \geq 0} \left( \E\abs{D_x f(X_t,u_t)}^r \right)^{1/r} < \infty;\qquad  \text{for some }r > 1. \]

Similar equations are studied in the formulation of the SMP for discounted cost functionals in infinite horizon, see e.g. \cite{maslowski2014sufficient}, \cite{orrieri2015necessary}. In that case, though, the spaces in which one is looking for a solution are weighted $L^2$-spaces, allowing the solution to explode at infinity in a controlled way. Here, due to the stability of the system, we expect the solution to be bounded up to infinity.    
 
The strategy we employ is mainly built on duality techniques. Via the analysis of a suitable class of perturbed linearized forward eqautions, see equation \eqref{eq:duality:general} below, we are able, exploiting their dissipativity,  to construct an adjoint process for all positive times. We introduce then a well-suited family of truncated equations and we show the consistency of the family with respect to the varying finite horizon $T>0$, as $T \to  \infty$.

We also propose a second version of maximum principle involving a
family of backward equations on finite time horizon $T$ with terminal condition $p^T_T = 0$  that could be verifiable in certain cases, see Remark \ref{thm:necessary truncated} below.

Once we have a necessary condition for optimality, it is natural to ask also for a sufficient counterpart of it. As in the classical setting, an extra convexity assumption on the Hamiltonian of the system guarantees the required sufficiency.   

The paper is structured as follows. In Section 2 we fix the notation and we discuss the main assumptions on the state equation and on the control actions. In Section 3 we study the convex perturbation of the optimal control and we expand the optimal trajectory and cost functional with respect to the perturbation. Section 4 is the core of the paper. Here we introduce the adjoint equation and we present a well-posedness result for it. The main results concerning the necessary and sufficient versions of the SMP are contained in Section 5 and 6.

\section{Preliminaries and assumptions}

Let $(\Omega, \Fcal, \mP)$ be a complete probability space and $(W_t)_{t\geq 0}$ a standard $d$-dimensional Brownian motion.
Throughout the paper we use the natural filtration $(\mathcal{F}_t)_{t\geq 0}$ associated to $W$, augmented in the usual way with the family of $\mP$-null sets of $\mathcal{F}$. By $|\cdot|$ we denote the Euclidean norm on $\R^n$ and $\norm{\cdot}_2$ denotes the Hilbert-Schmidt norm on $\R^{n \times n}$.

For any $p \geq 1$ and $T>0$ we define
\begin{itemize}
\item $L^p(\Omega\times [0,T]; \R^n)$, the set of all $(\mathcal{F}_t)$-progressive processes with values in $\R^n$ such that
\[ \norm{X}_{L^p(\Omega\times[0,T];\R^n)} = \left( \E \int_0^T \abs{X_t}^p dt\right)^{1/p} < \infty; \]
\item $L^p(\R_+; L^q(\Omega;\R^n))$ the set of all $(\mathcal{F}_t)$-progressive processes with values in $\R^n$ with $1\leq q < +\infty$ such that
\[ \norm{X}_{L^p(\R_+; L^q(\Omega;\R^n))}^p =  \int^{\infty}_0 \lp \E \abs{X_t}^q \rp^{\frac{p}{q}}dt < \infty,\quad \text{for } 1\leq p < +\infty, \]
and
\[ \norm{X}_{L^{\infty}(\R_+; L^q(\Omega;\R^n))} = \sup_{t \geq 0} \lp \E \abs{X_t}^q \rp^{\frac{1}{q}} < \infty.\]
\end{itemize}

\noindent The aim of this work is to give some necessary (and sufficient) condition for optimality of a controlled system of the form
\begin{equation}\label{SDE}
\begin{system}
dX_t = b(X_t,u_t)dt + \sigma(X_t,u_t)dW_t, \qquad t\geq 0,\\
X_0 = x,
\end{system}
\end{equation}when a cost functional of ergodic type has to be minimized. The form of the cost functional slightly differs when considering a $\liminf$ or a $\limsup$ formulation. 
We define a truncated cost functional in the following form
\begin{equation}\label{cost truncated}
J_T (u(\cdot)) = \E\int_0^T f(X_t,u_t)dt.
\end{equation}
Let us denote the two forms in the following way
\begin{equation}\label{eq:cost_inf}
J^{\inf} (u(\cdot)) = \liminf_{T \to \infty} \frac{1}{T} J_T(u(\cdot)) = \liminf_{T \to \infty} \frac{1}{T}\E\int_0^T f(X_t,u_t)dt,
\end{equation}
\begin{equation}\label{eq:cost_sup}
J^{\sup} (u(\cdot)) = \limsup_{T \to \infty} \frac{1}{T} J_T(u(\cdot)) = \limsup_{T \to \infty} \frac{1}{T}\E\int_0^T f(X_t,u_t)dt.
\end{equation}
An control process $\bar{u}(\cdot)$ is said to be optimal either if
\begin{equation}\label{ergodic cost}
J^{\inf} (\bar{u}(\cdot)) = \inf_{u(\cdot) \in \mathcal{U}} J(u(\cdot)) \qquad \text{ or } \qquad J^{\sup} (\bar{u}(\cdot)) = \inf_{u(\cdot) \in \mathcal{U}} J(u(\cdot)),
\end{equation}
where $\mathcal{U}$ indicates a class of admissible controls.
Now we give some assumptions on the state equation and on the control actions.

\begin{hypothesis}\label{Hyp} Assumptions involve three constants $m\geq 0$ and $p>(4m+2){\color{red}\vee 4}$ and $k>(p-1)/2$ 
 that we fix now and for the rest of the paper.
\begin{itemize}
\item[]
\item[]
\item[(H1)] ({\bf{Controls}}) $U$ is a closed convex subset of $\R^l$. Moreover $u$ is a progressively measurable $U$-valued process. We say that $u$ is an \textit{admissible control} if it satisfies:
\begin{equation}
\sup_{t \geq 0}\E\abs{u_t}^p < +\infty.
\end{equation}

\item[(H2)] ({\bf{Polynomial growth}}) The vector field $b: \R^n \times U \rightarrow \R^n$ is $\mathcal{B}(\R^n)\otimes \mathcal{B}(U)$-measurable and $\mathcal{C}^2$ with respect to $x$ and $u$. There exists $C_1 >0$ such that 
\[ \abs{D_ub(x,u)} \leq C_1, \qquad x \in \R^n, u \in U. \]
Moreover:
\begin{equation} \label{eq:grad b}
\sup_{u \in U}\sup_{x \in \mathbb{R}^n} \frac{\abs{D^{\beta}_x b(x,u)} }{1+ \abs{x}^{2m+1-\beta} + \abs{u}^{1-\beta}} < +\infty,\qquad \beta =0,1.
\end{equation}

\item[(H3)] ({\bf{Polynomial growth}}) The mapping $\sigma: \R^n \times U \rightarrow \R^{n\times d} $ is measurable with respect to $\mathcal{B}(\R^n) \otimes \mathcal{B}(U)$. There exists $C_2 >0$ such that 
\[ \norm{D_u\sigma(x,u)}_2 \leq C_2, \qquad x \in \R^n, u \in U. \]
Moreover it is $\mathcal{C}^2$ with respect to $x$, $u$ and:
\begin{equation} \label{eq:grad sigma}
\sup_{u\in U}\sup_{x \in \mathbb{R}^n} \frac{\norm{D^{\beta}_x \sigma(x,u)}_2 }{1+ \abs{x}^{m-\beta} + \abs{u}^{1-\beta}} < +\infty,\qquad \beta =0,1.
\end{equation}
\smallskip
\item[(H4)] ({\bf{Joint dissipativity}}) There is $c_p <0$ such that
\begin{equation} \label{eq:joint dissipativity}
\braket{D_xb(x,u)y,y} + k\norm{D_x\sigma(x,u)y}^2_2 \leq c_p \abs{y}^2, \qquad x,y \in \mathbb{R}^n, u \in U.
\end{equation}
\smallskip
\item[(H5)] ({\bf{Cost}}) The function $f: \R^n \times U \rightarrow \R$ is $\mathcal{B}(\R^n)\otimes \mathcal{B}(U)$-measurable, bounded from below by a constant $f_0$, it is differentiable in $x$ and $u$ and 
\[ \abs{D_xf(x,u)} + \abs{D_uf(x,u)}\leq C(1 + \abs{x} + \abs{u}),\]
for some $C >0$.
\end{itemize}
\end{hypothesis}

\begin{remark}
We refer to \cite{cerrai2001second} and \cite{orrieri2015necessary} for a discussion on the joint monotonicity and on the relation between the growth of $b$ and $\sigma$. Concerning (H5), here we limit ourselves to linear growth for simplicity. A general polynomial growth can be easily achieved.
\end{remark}

\begin{remark}
The choice of $p > 4m+2$ in (H1)-(H4) comes from the interplay between the dissipative behaviour of the system and polynomial growth of the coefficients. Actually this bound can be easily derived from the maximal moment of the state process that we need to estimate in the proofs (see Proposition \ref{l:convergence X tilde}). The condition for $k$ is then the natural one. 
\end{remark}
We can state the following

\begin{theorem}\label{t.existence.SDE}
Assume that Hypothesis \ref{Hyp} holds true. Then, for every $x \in \R^n$ and every admissible control $u(\cdot)$, equation \eqref{SDE} admits a unique progressively measurable solution for each admissible control. 
Moreover, the following estimate holds
\begin{equation}\label{eq:estimate.state}
\E\abs{X_t}^{p} \leq e^{-p\beta t}\abs{x}^{p} + K \sup_{t \geq 0}\E \abs{u_t}^{p},
\end{equation}
for some positive constants $K = K(p,c_p)$ and $\beta$.
\end{theorem}
\begin{proof} 
Define $\tilde{X}_t := e^{\beta t}X_t$ for a positive $\beta$. Then $\tilde{X}$ solves
\begin{equation}\label{eq:SDE tilde}
\begin{system}
d\tilde{X}_t = \beta\tilde{X}_t +  e^{\beta t} b( e^{-\beta t} \tilde{X}_t,u_t)dt
+ e^{\beta t}\sigma ( e^{-\beta t} \tilde{X}_t,u_t)dW_t, \qquad \forall t\geq 0,\\
\tilde{X}_0 = x.
\end{system}
\end{equation}
If we call $\tilde{b}_t(x,u) = e^{\beta t} b\bigl( e^{-\beta t} x, u \bigr)$ and
$\tilde{\sigma}_t(x,u) = e^{\beta t} \sigma \bigl( e^{-\beta t} x,u \bigr)$ then also
$\tilde{b_t}, \tilde{\sigma_t}$ satisfy Hypothesis \ref{Hyp}. In particular the joint dissipativity holds with the same constant
\begin{equation}
\braket{\tilde{b}_t(x,u) - \tilde{b}_t(y,u), x-y} + \frac{p-1}{2}\norm{\tilde{\sigma}_t(x,u) - \tilde{\sigma}_t(y,u)}^2_2
\leq c_p |x-y|^2.
\end{equation}
Let $p \geq 2$, denote $p = 2q$ and $\tilde a = \tilde \sigma(x,u)^*\tilde \sigma(x,u) $ (we omit the time dependence $\tilde \sigma = \tilde \sigma_t$ when it is clear). We apply the It\^o formula to the function $f(x) = \abs{x}^{2q}$ to get
\begin{equation*}
\begin{split}
\E &\abs{\tilde{X}_t}^{2q} = |x|^{2q}
+ 2q\E \int^t_0 \abs{\tilde{X}_s}^{2(q-1)}\left(\braket{\tilde{X}_s, \tilde{b}( \tilde X_s,u_s)}
+ \frac{1}{2} \norm{\tilde\sigma(\tilde X_s,u_s)}_2^2 \right) ds \\
&\quad+ 2q\beta\E\int_0^t\abs{\tilde{X}_s}^{2q}ds + 2q(q-1)\E \int^t_0  \abs{\tilde{X}_s}^{2(q-2)}Tr\lsz \tilde{a}_s \lp \tilde{X}_s \otimes \tilde{X}_s \rp \psz ds \\
&\leq |x|^{2q} + 2q\E \int^t_0 \abs{\tilde{X}_s}^{2(q-1)}\left(\braket{\tilde{X}_s, \tilde{b}( \tilde X_s,u_s)} +(q- \frac{1}{2}) \norm{\tilde\sigma(\tilde X_s,u_s)}_2^2 \right) ds + 2q\beta\E\int_0^t\abs{\tilde{X}_s}^{2q}ds\\
&\leq |x|^{2q} + 2q\E\int_0^t\abs{\tilde{X}_s}^{2(q-1)}\left(\braket{\tilde{X}_s, \tilde{b}(\tilde X_s,u_s) - \tilde{b}(0,u_s)} + (q -\frac{1}{2})(1+\varepsilon)\norm{\tilde\sigma(\tilde X_s,u_s) - \tilde\sigma(0,u_s)}_2^2 \right) ds \\
\end{split}
\end{equation*}
\begin{equation*}
\begin{split}
&\quad+2q\E\int_0^t\abs{\tilde{X}_s}^{2(q-1)}\left(\braket{\tilde{X}_s, \tilde{b}(0,u_s)}
+ c_\varepsilon\norm{\tilde\sigma(0,u_s)}_2^2 \right) ds + 2q\beta\E\int_0^t\abs{\tilde{X}_s}^{2q}ds\\
&\leq |x|^{2q} + 2q\left( c_{r} + \beta + \frac{\delta}{2}\right)\E\int_0^t\abs{\tilde{X}_s}^{2q}ds \qquad \qquad \qquad \qquad  \qquad \qquad(\text{ with } r = 2q(1+\varepsilon) - \varepsilon)\\
&\quad+2q\E\int_0^t\abs{\tilde{X}_s}^{2(q-1)}\left(\frac{1}{2\delta}\abs{\tilde{b}(0,u_s)}^2 + c_\varepsilon\norm{\tilde\sigma(0,u_s)}_2^2 \right) ds \\
&\leq |x|^{2q} + 2q\left( c_{r} + \beta + \frac{\delta}{2} + c_\delta\delta^{q/(q-1)} \right)\E\int_0^t\abs{\tilde{X}_s}^{2q}ds\\
&\quad+ 2q\E \int_0^t e^{-2q\beta t}\left(\frac{1}{2^q \delta^{q+1}q} + \frac{c_\varepsilon}{\delta^q} \right)\abs{u_s}^{2q}ds, \\
\end{split}
\end{equation*}
where we employed joint dissipativity for the process $\tilde X$, we repeatedly used weighted Young inequality and in the end the growth condition on the coefficients. Choosing $\beta$ and $\delta$ small enough, thanks to Hypothesis (H1) we end up with the following estimate
\begin{equation}
\begin{split}
\E\abs{X_t}^{2q} &\leq e^{-2q\beta t}\abs{x}^{2q} + C \int_0^t e^{-2q\beta(t-s)} \E \abs{u_s}^{2q}ds \\
&\leq e^{-2q\beta t}\abs{x}^{2q} + C \sup_{t \geq 0}\E \abs{u_t}^{2q}.
\end{split}
\end{equation}
Notice that, taking the supremum on both sides we also have that
\begin{equation}
\sup_{t \geq 0}\E\abs{X_t}^{2q} \leq C(\abs{x}^{2q} + 1),
\end{equation}
and the claim is proved.\end{proof}
\section{Perturbation of the controls}
When considering ergodic control problems we can not expect to gain information by the use of local in time perturbations of the optimal control. 

 More precisely, let ${u}^i(\cdot)$, $i=1,2$ are admissible controls with $u^1_t= {u}^2_t $ for all $t>T_0$. If one denotes by ${X}^i$, $i=1,2$ the corresponding states then by the dissipativity assumption (H4) one gets $\mathbb{E}|X^1_t-X^2_t|^2\rightarrow 0$ for $t>T_0$ exponentially fast (let us say with exponential decay $\varepsilon$).  Consequently (assume for a moment that $f$ is Lipschitz) 
 \begin{equation}
\begin{split}
| J(u^1(\cdot)) - J({u}^2(\cdot))| &=
\lim_{T \to \infty} \frac{1}{T}\E\int_0^T \left| f({X}^1_t,{u}^1_t) - f({X}^2_t,{u}^2_t)\right|dt \\
&=
\lim_{T \to \infty} \frac{1}{T}\E\int_{T_0}^T \left| f({X}^1_t,{u}^1_t) - f({X}^2_t,{u}^1_t)\right|dt \\
&\leq C \lim_{T \to \infty} \frac{1}{T}\int_{T_0}^T e^{-\varepsilon t} dt = 0,
\end{split}
\end{equation}
 
This is the reason for considering the perturbations which act on the system up to infinity. Notice that it is crucial to require that   $U$ is convex. \smallskip

Let then $ \bar{u}(\cdot)$ be an optimal control for the ergodic control problem \eqref{ergodic cost} and denote the corresponding state process as $\bar{X}$.
For  $\theta \in (0,1]$ and $u(\cdot)$ admissible control define $u^{\theta}$ as a convex combination by $u^{\theta}(\cdot): 
=
(1-\theta)  \bar{u}(\cdot) + \theta u(\cdot)=  \bar{u}(\cdot) + \theta v(\cdot)$, where $v(\cdot) := u(\cdot) - \bar{u}(\cdot)$.
Then $u^{\theta}(\cdot)$ is admissible and the corresponding state is denoted by $X^{\theta}$. 

\begin{lemma}\label{lemma:x epsilon}
Under Hypothesis \ref{Hyp} the following holds
\begin{equation*}
\operatorname{sup}_{t \geq 0} \E \abs{ X^{\theta}_t - \bar{X}_t }^p \leq C \theta^2 \operatorname{sup}_{t \geq 0}  \E \abs{ v_t }^p.
\end{equation*} 
Where $C$ only depends on the constants appearing in Hypothesis \ref{Hyp}.
\end{lemma}
\begin{proof}
Denote $\Delta X^{\theta}_t := X^{\theta}_t - \bar{X}_t$ and write the corresponding equation
\begin{equation}
\Delta X^{\theta}_t = \int_0^t \left[ b(X^{\theta}_s,u^{\theta}_s) - b(\bar{X}_s, \bar{u}_s) \right]ds + \int_0^t \left[ \sigma(X^{\theta}_s,u^{\theta}_s) - \sigma(\bar{X}_s, \bar{u}_s) \right]dW_s.
\end{equation}
Following the technique developed in the proof of Theorem \ref{t.existence.SDE} we define $\Delta \tilde{X}^\theta_t := e^{\beta t}\Delta X^\theta_t$ for a positive $\beta$. Then the It\^o formula gives
\begin{equation*}
\begin{split}
&\E\abs{ \Delta \tilde X^{\theta}_t}^{2q}\\
 &\leq  2q\E\int^t_0 \abs{\Delta \tilde X^{\theta}_s}^{2(q-1)}\lbr \braket{\tilde b(\tilde X^{\theta}_s, u^\theta_s) - \tilde b(\tilde{\bar{X}}_s,u^\theta_s) ,\Delta \tilde X^{\theta}_s}   + (q-\frac{1}{2})(1+\varepsilon)\norm{\tilde \sigma(\tilde X^{\theta}_s,u^{\theta}_s) - \tilde \sigma(\tilde{\bar{X}}_s, u^\theta_s)}_2^2\rbr ds\\
&\quad+ 2q \beta\E\int^t_0  \abs{\Delta \tilde X^{\theta}_s}^{2q}ds  + 2q \E\int^t_0  \abs{\Delta \tilde X^{\theta}_s}^{2(q-1)}\braket{ \int^1_0 D_u \tilde b( \tilde{\bar{X}}_s, \bar{u}_s + \lambda \theta v_s )\theta v_s d\lambda , \Delta \tilde X^{\theta}_s} ds \\
&\quad+ 2qc_\varepsilon\E \int^t_0  \abs{\Delta\tilde X^{\theta}_s}^{2(q-1)} \norm{\int_0^1 D_u\tilde \sigma(\tilde{\bar X}_s,\bar u_s + \lambda \theta v_s)\theta v_s d\lambda}_2^2 ds \\
&\leq 2q(c_r + \beta + \frac{\delta}{2})\int^t_0 \abs{ \Delta \tilde X^{\theta}_s }^{2q}ds + 2q\theta^2\E\int_0^t \abs{ \Delta \tilde X^{\theta}_s }^{2(q-1)}\frac{e^{2\beta s}}{2\delta}\abs{v_s}^2ds \quad \qquad(\text{ with } r = 2q(1+\varepsilon) - \varepsilon)\\ 
&\quad+ 2qc_\varepsilon\theta^2\E\int_0^t \abs{ \Delta \tilde X^{\theta}_s }^{2(q-1)}e^{2\beta s}\abs{v_s}^2ds \\
&\leq 2q\left(c_r + \beta + \frac{\delta}{2} + c_\delta \theta^2 \delta^{q/(q-1)}\right)\int^t_0 \abs{ \Delta \tilde X^{\theta}_s }^{2q}ds\\
&\quad+ 2q\theta^2\E\int_0^t e^{2q\beta s}\left( \frac{1}{2^q\delta^{q+1}}+ \frac{c_\varepsilon}{\delta^q} \right)\abs{v_s}^{2q} ds. 
\end{split}
\end{equation*}
Where we used the joint dissipativity and weighted Young inequality, for every $\delta >0$. Choosing $\beta$, $\delta$ small enough, from the boundedness of $\sup_{s\geq 0}\E\abs{v_s}^{2q}$ we get
\begin{equation*}
\E\abs{ \Delta X^{\theta}_t}^{2q} \leq C\theta^2 \E \int_0^t e^{-2q\beta (t-s)}ds.
\end{equation*}
The result follows by taking the supremum in time and finally by sending $\theta \rightarrow 0_+$.

\end{proof}

\noindent Now we introduce the first variation equation of the system. Notice that in the equation appears the derivative of the coefficients with respect to the control, which are bounded due to our assumptions.
\begin{equation}\label{eq:first:variation}
\begin{system}
dY_t = \left[ D_xb(\bar{X}_t,\bar{u}_t)Y_t + D_ub(\bar{X}_t,\bar{u}_t)v_t \right]dt + \left[ D_x\sigma(\bar{X}_t,\bar{u}_t)Y_t + D_u\sigma(\bar{X}_t,\bar{u}_t)v_t \right]dW_t,\\
Y_0 = 0,
\end{system}
\end{equation}

\begin{lemma}\label{lemma:y}
Under Hypothesis \ref{Hyp}, the first variation equation \eqref{eq:first:variation} admits a unique adapted solution. Moreover the following estimate holds true 
\begin{equation}\label{eq:pth moment first variation}
\E \lav Y_t \rav^p \leq K \sup_{s \in [0,t]}\E\abs{v_s}^p.
\end{equation}
where again $K$ only depends on the constants appearing in Hypothesis \ref{Hyp}.

In particular, $\sup_{t \geq 0}\E \lav Y_t \rav^p \leq K \sup_{t \geq 0}\E \lav v_t \rav^p <+\infty$.
\end{lemma} 
\begin{proof}
The proof goes through by the same technique adopted in Theorem \ref{t.existence.SDE}. What is crucial here is the uniform boundedness of $D_u b(x,u)$ and $D_u\sigma(x,u)$, along with the assumption (H1) on admissible controls.    
\end{proof}

The following lemma is fundamental in order to obtain the right expansion of the cost functional with respect to the control.
\begin{proposition}\label{l:convergence X tilde}
Under our assumptions the process $\hat{X}^{\theta}$ defined as
\begin{equation*}
\hat{X}^{\theta}_t = \dfrac{X^{\theta}_t - \bar{X}_t}{\theta} - Y_t,
\end{equation*}
satisfies
\begin{equation}
\lim_{\theta \rightarrow 0_+} \sup_{t \geq 0}\E \abs{ \hat{X}^{\theta}_t}^2 = 0.
\end{equation}
\end{proposition}

\begin{proof}
\noindent The equation for $\hat{X}^{\theta}$ reads
\begin{equation}
\begin{split}
d\hat{X}^{\theta}_t &= 
\frac{1}{\theta}\left[ b( X^{\theta}_t,u^{\theta}_t) - b(\bar{X}_t,\bar{u}_t) - \theta D_x b(\bar{X}_t,\bar{u}_t) Y_t - \theta D_u b(\bar{X}_t,\bar{u}_t) v_t \right]dt \nonumber \\
& \quad + \frac{1}{\theta}\left[ \sigma( X^{\theta}_t,u^{\theta}_t) - \sigma(\bar{X}_t,\bar{u}_t) - \theta D_x \sigma(\bar{X}_t,\bar{u}_t) Y_t - \theta D_u \sigma(\bar{X}_t,\bar{u}_t) v_t \right]dW_t \nonumber \\
&= \frac{1}{\theta}\left[ b\lp \bar{X}_t + \theta (Y_t + \hat{X}^{\theta}_t),\bar{u}_t+\theta v_t\rp - b(\bar{X}_t,\bar{u}_t ) - \theta D_x b(\bar{X}_t,\bar{u}_t) Y_t - \theta D_u b(\bar{X}_t,\bar{u}_t) v_t\right]dt \nonumber \\
&\quad + \frac{1}{\theta}\left[ \sigma \lp \bar{X}_t + \theta (Y_t + \hat{X}^{\theta}_t),\bar{u}_t+\theta v_t \rp - \sigma(\bar{X}_t,\bar{u}_t) - \theta D_x \sigma(\bar{X}_t,\bar{u}_t) Y_t - \theta D_u \sigma(\bar{X}_t,\bar{u}_t) v_t \right]dW_t,
\end{split}
\end{equation}
with $\hat{X}^{\theta}_0 = 0$ as initial condition. Further, by Taylor expansion we have that 
\begin{equation}
\begin{split}
d\hat{X}^{\theta}_t &= 
\int^1_0 D_x b \lp \bar{X}_t + \lambda \theta (Y_t + \hat{X}^{\theta}_t),\bar{u}_t + \lambda \theta v_t \rp \hat{X}^{\theta}_t d\lambda dt \nonumber\\ & \quad 
+ \int^1_0 \left[ D_x b \lp \bar{X}_t + \lambda \theta (Y_t + \hat{X}^{\theta}_t),\bar{u}_t + \lambda \theta v_t \rp - D_x b(\bar{X}_t,\bar{u}_t) \right] Y_t d\lambda dt \nonumber \\
& \quad + \int^1_0 \left[ D_u b \lp \bar{X}_t + \lambda \theta (Y_t + \hat{X}^{\theta}_t),\bar{u}_t + \lambda \theta v_t \rp - D_u b(\bar{X}_t,\bar{u}_t) \right] v_t d\lambda dt \nonumber \\
& \quad + \int^1_0 D_x \sigma \lp \bar{X}_t + \lambda \theta (Y_t + \hat{X}^{\theta}_t),\bar{u}_t + \lambda \theta v_t \rp \hat{X}^{\theta}_t d\lambda dW_t \nonumber\\ & \quad + \int^1_0 \left[ D_x \sigma \lp \bar{X}_t + \lambda \theta (Y_t + \hat{X}^{\theta}_t),\bar{u}_t + \lambda \theta v_t \rp 
- D_x \sigma(\bar{X}_t,\bar{u}_t) \right] Y_t d\lambda dW_t \nonumber \\
& \quad + \int^1_0 \left[ D_u \sigma \lp \bar{X}_t + \lambda \theta (Y_t + \hat{X}^{\theta}_t),\bar{u}_t + \lambda \theta v_t \rp 
- D_u \sigma (\bar{X}_t,\bar{u}_t) \right] v_t d\lambda dW_t.
\end{split}
\end{equation}
To keep the notation simple, we rewrite the above equation as
\begin{equation*}
d\hat{X}^{\theta}_t = 
\lp A^x_t \hat{X}^{\theta}_t + A^y_t Y_t + A^v_t v_t \rp dt + \lp B^x_t \hat{X}^{\theta}_t  + B^y_t Y_t + B^v_t v_t \rp dW_t,
\end{equation*}
where we have kept the order of the terms from the previous equation.\\
Now apply the It\^o formula to $e^{\beta t}\big| \hat{X}^{\theta}_t \big|^2$ to get 
\begin{equation}\label{eq:x tilde Ito}
\begin{split}
\E \big( e^{\beta t}\big| \hat{X}^{\theta}_t \big|^2 \big) &= 
2\E \int^t_0 e^{\beta s}
\big< A^x_s \hat{X}^{\theta}_s + A^y_s Y_s + A^v_s v_s, \hat{X}^{\theta}_s\big> ds \\
& + \E \int^t_0 e^{\beta s}  \norm{ B^x_s \hat{X}^{\theta}_s + B^y_s Y_s + B^v_s v_s}_2^2 ds  + \beta \E \int^t_0 e^{\beta s}\big| \hat{X}^{\theta}_s \big|^2 ds.
\end{split}
\end{equation}
By the  joint dissipativity assumption (H4) in Hypothesis \ref{Hyp} we have
$$  
2\big< A^x_s \hat{X}^{\theta}_s, \hat{X}^{\theta}_s \big> +2 k  \norm{ B^x_s \hat{X}^{\theta}_s }^2 +\beta | \hat{X}^{\theta}_s|^2 <0,$$
for some $k>1/2$ and $\beta$ small enough.

Thus, repeating the same computations as in the proof of Theorem \ref{t.existence.SDE}, we get the following intermediate estimate
\begin{equation}\label{eq:x tilde prefinal}
\E \big| \hat{X}^{\theta}_t \big|^2 \leq
C \int^t_0 e^{-\beta(t-s)} \E \lp
 \abs{A^y_s Y_s}^2 + \abs{A^v_s v_s}^2 +  \abs{B^y_s Y_s}^2 + \abs{B^v_s v_s}^2 \rp ds.
\end{equation}
Now we show how to treat the first term in \eqref{eq:x tilde prefinal}.  The estimate of the remaining ones goes along similar lines.

 We fix $\alpha$ with $p/(p-2)< \alpha <p/(4m)$, if $m\geq 1$, or $\alpha=2$, if $m=0$. Recall that $p> 4m+2$ and notice that, this way, denoting by $\alpha'$ the conjugate of $\alpha$ (that is $1/\alpha+1/\alpha'=1$) then $2 \alpha '<p$  $4m\alpha <p$ and $2\alpha <p$.  First we start by observing that by H\"older inequality and by \eqref{eq:pth moment first variation} we have that for any $\alpha>1$.
\begin{equation}\label{eq:x tilde loc lipsch1}
\begin{split}
&\int^t_0 e^{-\beta(t-s)} \E \abs{A^y_s Y_s}^2 ds \\
&= \int^t_0 e^{-\beta(t-s)} \E \lav \int^1_0 \left[ D_x b \lp \bar{X}_s + \lambda \theta (Y_s + \hat{X}^{\theta}_s),\bar{u}_s + \lambda \theta v_s \rp - D_x b(\bar{X}_s,\bar{u}_s) \right] Y_s d\lambda \rav^2 ds \\
&\leq \int^t_0\!\! e^{-\beta(t-s)}\lp \int^1_0 
 \E \lav   D_x b \lp \bar{X}_s + \lambda \theta (Y_s + \hat{X}^{\theta}_s),\bar{u}_s + \lambda \theta v_s \rp - D_x b(\bar{X}_s,\bar{u}_s)\rav^{2\alpha}d\lambda  \rp^{\frac{1}{\alpha}} 
 \!\!\! \cdot \lp \E |Y_s|^{2\alpha'} \rp ^{\frac{1}{\alpha'}}   ds
 \end{split}
\end{equation} 
Since $2\alpha'<p$, using Lemma \ref{lemma:y} to estimate $\sup_{s\in \mathbb{R}^+} \E |Y_s|^{2\alpha'} $:
\begin{equation}
\begin{split}
&\int^t_0 e^{-\beta(t-s)} \E \abs{A^y_s Y_s}^2 ds \\
&\leq C \int^t_0 e^{-\beta(t-s)} \lp \int^1_0
 \E \lav   D_x b \lp \bar{X}_s + \lambda \theta (Y_s + \hat{X}^{\theta}_s),\bar{u}_s + \lambda \theta v_s \rp - D_x b(\bar{X}_s,\bar{u}_s)\rav^{2\alpha}   d\lambda\rp^{\frac{1}{\alpha}}  ds.\\
&\leq C \int^t_0 e^{-\beta(t-s)} \lp \int^1_0
 \E \lav   D_x b \lp \bar{X}_s + \lambda \theta (Y_s + \hat{X}^{\theta}_s),\bar{u}_s + \lambda \theta v_s \rp - D_x b(\bar{X}_s,,\bar{u}_s + \lambda \theta v_s)\rav^{2\alpha} \!\! d\lambda  \rp^{\frac{1}{\alpha}} \!\!\!\!  ds\\  & \quad + C \int^t_0 e^{-\beta(t-s)} \lp \int^1_0
 \E \lav   D_x b \lp \bar{X}_s,\bar{u}_s + \lambda \theta v_s \rp - D_x b(\bar{X}_s,\bar{u}_s)\rav^{2\alpha} \!\!  d\lambda  \rp^{\frac{1}{\alpha}} \!\!\!\!  ds.
\end{split}
\end{equation}
We prove convergence of the first term, being the second similar (and easier).

Due to Hypothesis \ref{Hyp}, the gradients $D_xb$ are locally Lipschitz functions with respect to $x$, so that for all $R > 0$ there exists $C_R$ such that $D_xb$ is Lipschitz with constant $C_R$ in the ball of radius $R$. For each $t$ and $\theta$ we define the sets
\begin{equation}
A_{t,\theta}(R) = \lbrace w \in \Omega: \abs{\bar X_t} > R \rbrace \cup \lbrace w \in \Omega: \abs{X^\theta_t} > R \rbrace.
\end{equation}
By Chebyshev inequality we know that
\begin{equation}\label{stimadiA}
\mP(A_{t,\theta}(R)) \leq \frac{\E\abs{\bar X_t}^2}{R^2} + \frac{\E\abs{X^\theta_t}^2}{R^2} \leq \frac{C}{R^2}, \qquad \forall\, t, \forall \, \theta.
\end{equation}
Denoting for simplicity $X^{\lambda}_s=\bar{X}_s+\lambda \theta (Y_s+\hat{X}_s))=(1-\lambda)\bar{X}_s+\lambda X^{\theta}_s$ we have
\begin{equation}\label{eq:x tilde loc lipsch2}
\begin{split}
\int^t_0 & e^{-\beta(t-s)} \lp \int^1_0
 \E \lav   D_x b \lp X^{\lambda}_s,\bar{u}_s + \lambda \theta v_s \rp - D_x b(\bar{X}_s,\bar{u}_s + \lambda \theta v_s)\rav^{2\alpha}  \!\! d\lambda \rp^{\frac{1}{\alpha}} \!\!ds \\
&\leq C \int^t_0 e^{-\beta(t-s)} \lp\int^1_0
 \int_{A_{s,\theta}(R)} \lav D_x b \lp X^{\lambda}_s,\bar{u}_s + \lambda \theta v_s \rp - D_x b(\bar{X}_s,\bar{u}_s + \lambda \theta v_s)\rav^{2\alpha}\!\! d\mP d\lambda  \right)^{\frac{1}{\alpha}}\!\!\! ds\\
&\quad +  C \int^t_0 e^{-\beta(t-s)} \lp \int_0^1 \int_{A_{s,\theta}^c(R)} \lav D_x b \lp X^{\lambda}_s,\bar{u}_s + \lambda \theta v_s \rp - D_x b(\bar{X}_s,\bar{u}_s + \lambda \theta v_s) \rav^{2\alpha}\!\! d\mP d\lambda
\rp^{\frac{1}{\alpha}} \!\! ds \\
&\leq C\int^t_0 e^{-\beta(t-s)} \left( \int_0^1  \mP(A_{s,\theta}(R))^{\frac{\delta}{1+\delta}}  \E \lav   D_x b \lp X^{\lambda}_s,\bar{u}_s + \lambda \theta v_s \rp - D_x b(\bar{X}_s,\bar{u}_s + \lambda \theta v_s)\rav^{2\alpha(1+\delta)}\!\!\! d\lambda \right)^{\frac{1}{\alpha(1+\delta)}} \!\!\! \!\!\!  ds  \\
&\quad + C\int^t_0 e^{-\beta(t-s)} C_R^{\frac{1}{\alpha}} \left(   \E \abs{X^\theta_s - \bar X_s}^{2\alpha} \right)^{\frac{1}{\alpha}} ds,
\end{split}
\end{equation}
where $\delta>0$ is such that $4m\alpha(1+\delta)\leq p$. \\

\noindent Fixed $\varepsilon >0$ we know  by \eqref{stimadiA} that there exists $R$ large enough so that $\mP(A_{s,\theta}(R)) \leq  \varepsilon$. Moreover by Hypothesis \ref{Hyp}, Theorem \ref{t.existence.SDE} and Lemma \ref{lemma:x epsilon}:
$$ \E \lav   D_x b \lp X^{\lambda}_s,\bar{u}_s + \lambda \theta v_s \rp - D_x b(\bar{X}_s,\bar{u}_s + \lambda \theta v_s)\rav^{2\alpha(1+\delta)}\leq 
C \left( \mathbb{E} |  X^{\lambda}_s |^{4m\alpha(1+\delta)}+ \mathbb{E} |   \bar{X}_s |^{4m\alpha(1+\delta)}\right)\leq C, $$
(if $m=0$ the above relation is straight forward). Thus the first of the two integrals in the last two lines  in (\ref{eq:x tilde loc lipsch2})  can be estimated, for $R$ large enough and all $\theta$, $\lambda$ in $[0,1]$, by $C \varepsilon ^{\delta/[\alpha(1+\delta)^2]}$.

\noindent Moreover, due to Lemma \ref{lemma:x epsilon} we have that
$
\sup_{t \geq 0} \E \abs{ X^{\theta}_t - \bar{X}_t }^p \rightarrow 0$  as $\theta\rightarrow 0$.

 Combining the two estimates above we have:
\begin{equation}
\sup_{t \geq 0} \int_0^t e^{-\beta(t-s)}\E\abs{A^y_sY_s}^2 ds  \rightarrow 0 \hbox{  as }\theta\rightarrow 0.
\end{equation}
Repeating the argument for all the terms in \eqref{eq:x tilde prefinal} we get the required result.
\end{proof}

\begin{remark}
Notice that we estimate only the second moment of the error term, uniformly in time. Nevertheless, estimate of higer moments of the the state and first variation process are needed in order to complete the proof. More precisely, we can tune the value of $\alpha$ in \eqref{eq:x tilde loc lipsch1} in order to minimize the maximal moment of the state equation we need to control. Indeed, the growth of  the first term is 
\begin{equation*}
\E \lav   D_x b \lp \bar{X}_s + \lambda \theta (Y_s + \hat{X}^{\theta}_s),\bar{u}_s + \lambda \theta v_s \rp - D_x b(\bar{X}_s,\bar{u}_s)\rav^{2\alpha}
\leq C \E \abs{\bar{X}_s}^{4m\alpha}.
\end{equation*}
So that, $4m\alpha = 2\alpha' = 2 \frac{\alpha}{\alpha - 1}$, from which $\alpha = \frac{2m+1}{2m}$. The maximal moment is then $p = 4m\alpha = 2(2m+1)$, which is the one appearing in Hypothesis \ref{Hyp}.
\end{remark}

\subsection{Perturbation of the cost} Due to the hypotheses on the admissible controls and the estimate \eqref{eq:estimate.state}  the cost is well posed:
\[ \liminf_{T \to \infty} \frac{1}{T}\E\int_0^T f(X_t,u_t)dt \leq K \left[1+ \sup_{t \geq 0}\E\abs{X_t}^2  + \sup_{t \geq 0}\E\abs{u_t}^2\right] < \infty. \]
The same is true for the $\limsup$ formulation.
The expansion of the functional with respect to a convex perturbation of the control is given in the following

\begin{lemma}\label{l:derivative_J} Let $\bar{u}$ be an optimal control and let $u$ be any admissible control. Letting $v=u-\bar{u}$ and using the above notation  the following holds:
\begin{equation}\label{eq:Gateaux J inf}
\lim_{\theta \rightarrow 0_+}\frac{J^{\inf}(\bar{u}(\cdot) + \theta v(\cdot)) - J^{\inf}(\bar{u}(\cdot))}{\theta} \leq \limsup_{T \to \infty} \frac{1}{T} \E\int_0^T \left[\braket{ D_xf(\bar{X}_t,\bar{u}_t), Y_t}_{\mathbb{R}^n}+ \braket{ D_uf(\bar{X}_t,\bar{u}_t), v_t }_{\mathbb{R}^l}  \right]dt,
\end{equation}
and 
\begin{equation}\label{eq:Gateaux J sup}
\lim_{\theta \rightarrow 0_+}\frac{J^{\sup}(\bar{u}(\cdot) + \theta v(\cdot)) - J^{\sup}(\bar{u}(\cdot))}{\theta} \leq \limsup_{T \to \infty} \frac{1}{T} \E\int_0^T \left[\braket{ D_xf(\bar{X}_t,\bar{u}_t), Y_t}_{\mathbb{R}^n}+ \braket{ D_uf(\bar{X}_t,\bar{u}_t), v_t }_{\mathbb{R}^l}  \right]dt.
\end{equation}
\end{lemma}

\begin{proof}
We prove the first relation. The proof of the second one goes along the same
lines. Let us compute

\begin{equation*}
\begin{split}
&\frac{J_T(\bar{u}(\cdot) + \theta v(\cdot)) - J_T(\bar{u}(\cdot))}{\theta} = \dfrac{1}{\theta} \E\int_0^T \bigr[ f(X^{\theta}_t,\bar{u}_t + \theta v_t) - f(\bar{X}_t,\bar{u}_t) \bigl]dt \\
&= \E \int_0^T \int_0^1 D_xf\big(\bar{X}_t + \lambda( X^{\theta}_t - \bar{X}_t),\bar{u}_t+ \lambda\theta v_t\big) \big(\hat{X}^{\theta}_t + Y_t\big) d\lambda dt \\
&\quad+ \E \int_0^T \int_0^1 D_u f\big(\bar{X}_t + \lambda( X^{\theta}_t - \bar{X}_t),\bar{u}_t+ \lambda\theta v_t\big) v_t d\lambda dt. \\
\end{split}
\end{equation*}
Passing to the ergodic $\liminf$ cost functional \eqref{cost inf} we have that
\begin{equation*}
\begin{split}
 & \frac{J^{\inf}(\bar{u}(\cdot) + \theta v(\cdot)) - J^{\inf}(\bar{u}(\cdot))}{\theta} = \frac{1}{\theta}\left[\liminf_{T \to \infty} \frac{1}{T}J_T(\bar{u}(\cdot) + \theta v(\cdot)) - \liminf_{T \to \infty}\frac{1}{T} J_T(\bar{u}(\cdot))\right]\\
&\quad \leq  \limsup_{T \to \infty} \frac{1}{T} \left[ \frac{J_T(\bar{u}(\cdot) + \theta v(\cdot)) - J_T(\bar{u}(\cdot))}{\theta}\right]\\
&\quad= \limsup_{T \to \infty} \frac{1}{T} \E \int_0^T \int_0^1 \big< D_x f\big(\bar{X}_t + \lambda( X^{\theta}_t - \bar{X}_t),\bar{u}_t+ \lambda\theta v_t\big), \hat{X}^{\theta}_t + Y_t \big>d\lambda dt \\
&\quad \quad+ \limsup_{T \to \infty} \frac{1}{T} \E \int_0^T \int_0^1  \big< D_u f\big(\bar{X}_t + \lambda( X^{\theta}_t - \bar{X}_t),\bar{u}_t+ \lambda\theta v_t\big), v_t  \big>_U d\lambda dt,
\end{split}
\end{equation*}
where we used that $\limsup (a_n) - \limsup (b_n) \leq \limsup (a_n - b_n)$, for $(a_n)_{n \geq 1}$ and  $(b_n)_{n \geq 1}$ two general real sequences.
The extra term can be estimated by
\begin{equation*}
\begin{split}
&\limsup_{T \to \infty} \frac{1}{T} \E\int_0^T \int_0^1 \big< D_x f\big(\bar{X}_t + \lambda( X^{\theta}_t - \bar{X}_t),\bar{u}_t+ \lambda\theta v_t\big), \hat{X}^{\theta}_t \big> d\lambda dt \\
&\qquad \leq \limsup_{T \to \infty} \frac{1}{T}\int_0^T \int_0^1(\E\abs{D_x f\big(\bar{X}_t + \lambda( X^{\theta}_t - \bar{X}_t),\bar{u}_t+ \lambda\theta v_t\big)}^2)^{1/2} (\E\abs{\hat{X}^{\theta}_t}^2)^{1/2} d\lambda dt, \\
\end{split}
\end{equation*}
which converges to zero, uniformly in $T$, as $\theta \to 0_+$. In fact, this follows from the linear growth of $D_x f(\cdot)$, the a priori estimates on $X_t$ and Lemma \ref{l:convergence X tilde}. The conclusion now easily follows.
\end{proof}

\section{The adjoint equation}

In this section we introduce the dual equation associated to the system, which is an infinite horizon Backward SDE in $\R^n$. Different approaches have been developed in the literature to study this class of equations. Here we present a duality method built on construction of a family of truncated equations and associated with a consistency argument. 
\noindent More precisely, the infinite horizon backward equation has the form
\begin{equation}\label{eq:BSDE}
-dp_t = \left[ D_xb(X_t,u_t)^*p_t + \sum_{i=1}^dD_x \sigma^i(X_t,u_t)^*q^i_t - D_xf(X_t,u_t)\right]dt - \sum_{i=1}^dq^i_tdW^i_t,
\end{equation}
where, fixed any orthonormal basis $(e_i)_{i=1,..d}$ in $\mathbb{R}^d$ we set $W^i_s=<e_i,W_s>$ and $\sigma^i(x,u)= \sigma(x,u)e_i$ moreover we denote by $(\cdot)^*$ the transposition operation in $\mathcal{L}(\mathbb{R}^n)$.
For every $T >0$ fixed, its solution has to be understood as
\begin{equation}\label{eq:BSDE_int}
\begin{split}
p_t &= p_T + \int_t^T \left[ D_xb(X_s,u_s)^*p_s +\sum_{i=1}^d D_x\sigma^i(X_s,u_s)^*q^i_s + D_xf(X_s,u_s)\right]ds - \sum_{i=1}^d \int_t^T q^i_sdW^i_s.
\end{split}
\end{equation}
where $p$ and $q^i$, $i=1,...,d$ take values in $\mathbb{R}^n$.
Due to Hypothesis \ref{Hyp} and estimate \eqref{eq:estimate.state} the forcing term in the driver is no better than bounded, so that 
$D_xf(X_s,u_s) \in L^{\infty}\lp \R_+; L^2(\Omega;\R^n) \rp$. Therefore we cannot expect the solution of \eqref{eq:BSDE} to be integrable up to infinity but only that $p \in L^{\infty}\lp \R_+; L^2(\Omega;\R^n) \rp$. Up to the authors"1¤7 knowledge, there is not a general wellposendess result for such multidimensional BSDE's. Partial results have been obtained in \cite{cohen2015classical} by a Girsanov argument that seems to work only if one knows a-priori that $\sum_{i=1}^d D_x\sigma^i(X_s,u_s)^*q^i_s$ can be written as $\sum_{i=1}^d q^i_s f^i$ for suitable adapted real  process $(f_i)_{i=1,...,d}$. In particular this is the case when $n = 1$ or the noise is additive. \bigskip

\noindent Here the solution will  be obtained via the introduction of a family of time truncations:
\begin{equation}\label{eq:BSDE_trunc}
\begin{system}
-dp^{T,\nu}_t = \left[ D_xb(X_t,u_t)^*p^{T,\nu}_t + \sum_{i=1}^dD_x\sigma^i(X_t,u_t)^*q^{i,T,\nu}_t - D_xf(X_t,u_t)\right]dt - 
\sum_{i=1}^dq^{i,T,\nu}_t dW^i_t,\\
p_T^{T,\nu} = \nu.
\end{system}
\end{equation}
which will be estimated by duality. 
For the approximating equation \eqref{eq:BSDE_trunc} a wellposedness result has been already adressed in \cite{orrieri2015necessary}. \\
\smallskip
\noindent To shorten the notation in the following paragraphs, let us denote 
\[\Lambda_t := D_xb(X_t,u_t), \quad \Gamma^i_t := D_x\sigma^i(X_t,u_t), \quad \Psi_t := D_xf(X_t,u_t);\]  moreover when $\nu=0$ the solution of equation \eqref{eq:BSDE_trunc} will be denoted by $(p^T,q^{i,T})$.

\begin{theorem}\label{t:truncatedBSDEs}
For all $T\geq 0$ and all $\nu\in L^2(\Omega,\mathcal{F}_T,\mathbb{P};\mathbb{R}^n)$ there exists a unique $(d+1)$-tuple of $\mathbb{R}^n$-valued, adapted processes $(p^{T,\nu },q^{1,T,\nu},...,q^{d,T,\nu})$ such that 
$p^{T,\nu }$ has continuous trajectories,
$\sup_{t\in [0,T]}\mathbb{E}|p^{T,\nu }_t|^2 +
\sum_{i=1}^d \mathbb{E}\int_0^T |q^{i,T,\nu}_t|^2 dt <\infty$ and, $\mathbb{P}$-almost surely, for all $t\in [0,T]$ it holds:
$$p^{T,\nu}_t=\nu+\int_t^T \Lambda^*_s p^{T,\nu}_sds+\sum_{i=1}^d\int_t^T (\Gamma^{i}_s)^* q^{i,T,\nu}ds+\int_t^T\Psi_sds+\sum_{i=1}^d \int_t^T q^{i,T,\nu}_s dW^i_s$$.
\end{theorem}

Consider now the following affine forward SDE with general forcing term $(\gamma,\rho^i)_{i=1,..,d}$ with $\gamma$ and $\rho^i$, $i=1,..,d$ in $L^2([0,T];L^2(\Omega;\R^{n}))$ and initial condition $\eta \in L^2(\Omega,\mathcal{F}_t;\R^n)$:
\begin{equation}\label{eq:duality:general}
\begin{system}
d\Y_s^{t,\eta,\gamma,\rho} = \Lambda_s \Y_s^{t,\eta,\gamma,\rho}dt + \sum_{i=1}^d \Gamma^i_t\Y_s^{t,\eta,\gamma,\rho}dW^i_t +\gamma_sds +\sum_{i=1}^d  \rho^i_s dW^i_s,\ s \geq t,\\
\Y_t^{t,\eta,\gamma,\rho} = \eta.
\end{system}
\end{equation}
Then by the same technique we adopted in the proof of Theorem \ref{t.existence.SDE}, the above equation admits a unique adapted solution and
\begin{equation}\label{stimadiY}
\E\abs{\Y_r^{t,\eta,\gamma,\rho}}^2 \leq e^{-2\beta(r-t)}\E\abs{\eta}^2 + K \int_t^r e^{-2\beta(r-s)}\E\left[|\gamma_s|^2+\abs{\rho^1_s}^2+....+\abs{\rho^d_s}^2\right] ds.
\end{equation}
When $\gamma\equiv 0$ then the solution to the above equation will be denoted by $\Y^{t,\eta,\rho}$ and when $\rho\equiv 0$ as well, it will be denoted by $\Y^{t,\eta,\gamma}$.
\smallskip

\noindent The next result is proven in \cite{orrieri2015necessary} by computing the It\^o formula the differential of the product $d\braket{Y^{t,\eta,\gamma,\rho}_s,p^{T,\nu}_s}$

\begin{lemma}\label{l:dualityFO} Given $(\rho^i)_{i=1,..,d}$ with $\gamma, \rho^i \in L^2([0,T];L^2(\Omega;\R^{n}))$, $\eta \in L^2(\Omega,\mathcal{F}_t;\R^n)$,  $\nu \in L^2(\Omega,\mathcal{F}_T;\R^n)$ it holds:
\begin{equation}\label{general_duality}
\E\int_t^T \braket{p^{T,\nu}_s, \gamma_s}ds+ 
\sum_{i=1}^d \E\int_t^T \braket{q^{i,T,\nu}_s, \rho^i_s}ds + \E \braket{p^{T,\nu}_t,\eta} = \E\int_t^T \braket{\Y_s^{t,\eta,\gamma,\rho}, \Psi_s}ds + \E\braket{\nu, \Y_T^{t,\eta,\gamma,\rho}}.
\end{equation}
\end{lemma}
\noindent In the following, relation \eqref{general_duality} will be the main instrument to get information on the behaviour of the BSDE. We will specifically choose the values of $t,\eta,\rho$ according to our needs. \\
\smallskip
\noindent We are now in a position to define the solution to the infinite horizon multidimensional BSDE and prove its existence and uniqueness

\begin{definition}  A solution to equation \eqref{eq:BSDE} is a  $(d+1)$-tuple of $\mathbb{R}^n$-valued, adapted processes $(p_t,q^1_t,...,q^d_t)_{t\in [0,\infty[}$ such that, for all $T>0$ and all $i=1,...,d$ it holds
$ \mathbb{E}\int_0^T |q^{i}_t|^2 dt <\infty$. Moreover $p$ has continuous trajectories and 
$\sup_{t\in [0,\infty)}\mathbb{E}|p_t|^2<\infty$. Finally, for all $0\leq t \leq T$, \eqref{eq:BSDE_int} holds $\mathbb{P}$-almost surely.
\end{definition}

 The main result of this section is the following
\begin{theorem}\label{t.existence_BSDE}
Let Hypothesis \ref{Hyp} holds true. Then equation \eqref{eq:BSDE} admits a unique solution  $(p^{\infty},$ $q^{1,\infty},...,q^{d,\infty})$.
\end{theorem}
\begin{proof} $ $

\noindent \textit{Existence:}
Let  in \eqref{general_duality} $\nu \equiv 0,\ \gamma \equiv 0,\ \rho \equiv 0$, $\eta \in L^2(\Omega,\mathcal{F}_t;\R^n)$ then 
\begin{equation}
\E \braket{p^T_t,\eta} = \E \int_t^T \braket{\Y^{t,\eta}_s ,\Psi_s } ds.
\end{equation}
Since $\Psi \in  L^{\infty}\lp \R_+; L^2(\Omega;\R^n) \rp$  by \eqref{stimadiY} we deduce that 
$$\E \int_t^T \braket{\Y^{t,\eta}_s, \Psi_s } ds \rightarrow \E \int_t^{\infty} \braket{\Y^{t,\eta}_s, \Psi_s} ds$$
and that the right hand side is a bounded linear operator from $L^2(\Omega,\mathcal{F}_t;\R^n) \to \R$.
Hence, by Riesz representation theorem there exists an element $P(t)\in L^2(\Omega,\mathcal{F}_t;\R^n)$ such that  
\begin{equation}\label{eq:Riesz:p_infty}
\E \braket{P(t),\eta} = \E \int_t^\infty \braket{\Y^{t,\eta}_s, \Psi_s } ds.
\end{equation}
Moreover $p^T(t)\rightharpoonup P(t)$ in $L^2(\Omega,\mathcal{F}_t,\mathbb{P};\mathbb{R}^n)$ and
$
\E\abs{P(t)}^2 \leq \beta^{-1} \sup_{s\in [0,\infty[} (\mathbb{E}|\Psi_s|^2)^{1/2}$ for all $t>0$.

Let now for all $N\in \mathbb{N}$, $(\tilde{p}^N_t,\tilde{q}^{1,N}_t,...,\tilde{q}^{d,N}_t)_{t\in [0,N]}$ be the solution of equation 
\eqref{eq:BSDE_trunc} with $T=N$ and $\nu=P(N)$.

We claim that, for all $N,M\in \mathbb{N}$ with $0\leq N\leq M$ and all $t\leq N$ it holds  \begin{equation}\label{eq:consistency}\tilde{p}^N(t)=\tilde{p}^M(t), \hbox{ $\mathbb{P}$-a.s.}\end{equation}
By definition 
and Lemma \ref{l:dualityFO} we deduce that for all $\eta\in L^2(\Omega,\mathcal{F}_t,\mathbb{P};\mathbb{R}^n)$
$$
 \E \braket{\tilde{p}^{N}_t, \eta} = \E\int_t^N \braket{\Y_s^{t,\eta}, \Psi_s}ds + \E\braket{P(N), \Y_N^{t,\eta}}.$$
 Plugging \eqref{eq:Riesz:p_infty} with $t=N$ in the above relation we have
 $$
 \E \braket{\tilde{p}^{N}_t, \eta} = \E\int_t^N \braket{\Y_s^{t,\eta}, \Psi_s}ds +  \E\int_N^{\infty} \braket{\Y_s^{N,\Y_N^{t,\eta}}, \Psi_s}ds.$$
and finally, observing that by uniqueness of the solution to equation \eqref{eq:duality:general} $\Y_s^{N,\Y_N^{t,\eta}}=\Y_s^{t,\eta}$ $\mathbb{P}$-a.s. we conclude
 $$
 \E \braket{\tilde{p}^{N}_t, \eta} = \E\int_t^{\infty} \braket{\Y_s^{t,\eta}, \Psi_s}ds=\mathbb{E}\braket{\eta, P(t)}.$$
 and our claim is proved since the right hand side does not depend on $N$.
 We also remark that by the above identity we deduce that 
 $$\sup_{t\in [0,N]} |\tilde{p}^N_t|^2 \leq  \beta^{-1} \sup_{s\in [0,\infty[} (\mathbb{E}|\psi_s|^2)^{1/2},$$
 and that the right hand side does not depend neither on $t$ nor on $N$.
 \bigskip
 
\noindent Now we define
 $$
 p^{\infty}_t=\sum_{N=1}^{\infty} \tilde{p}^N_t I_{[N-1,N[}(t),\quad 
 q^{i,\infty}_t=\sum_{N=1}^{\infty} \tilde{q}^{i,N}_t I_{[N-1,N[}(t),$$
 and claim that it is the desired solution.
 Indeed it satisfies the desired integrability and adaptedness conditions. Moreover fixed $0\leq t\leq T$ then
 \begin{equation}\begin{split}p^{\infty}_t-p^{\infty}_T &=
[ p^{\infty}_t-p^{\infty}_{\lfloor t\rfloor+1}]+ [ p^{\infty}_{\lfloor T\rfloor }-p^{\infty}_T]+
 \sum_{n=\lfloor t\rfloor+1}^{\lfloor T\rfloor -1}[ p^{\infty}_n-p^{\infty}_{n+1}] \\
 &=\big[ \tilde{p}^{\lfloor t\rfloor+1}_t-\tilde{p}^{\lfloor t\rfloor+2}_{\lfloor t\rfloor+1}\big]+ \big[ \tilde{p}^{\lfloor T\rfloor+1}_{\lfloor T\rfloor }-\tilde{p}^{\lfloor T\rfloor+1}_T\big]+
 \sum_{n=\lfloor t\rfloor+1}^{\lfloor T\rfloor -1}[ \tilde{p}^{n+1}_n-\tilde{p}^{n+2}_{n+1}]
  \\
 &=\big[ \tilde{p}^{\lfloor t\rfloor+1}_t-\tilde{p}^{\lfloor t\rfloor+1}_{\lfloor t\rfloor+1}\big]+ \big[ \tilde{p}^{\lfloor T\rfloor+1}_{\lfloor T\rfloor }-\tilde{p}^{\lfloor T\rfloor+1}_T\big]+
 \sum_{n=\lfloor t\rfloor+1}^{\lfloor T\rfloor -1}[ \tilde{p}^{n+1}_n-\tilde{p}^{n+1}_{n+1}],
 \end{split}\end{equation}
 where in the last equality we have exploited \eqref{eq:consistency} where it was needed. Now recalling that $(\tilde{p}^N_t,\tilde{q}^{1,N}_t,...,\tilde{q}^{d,N}_t)_{t\in [0,N]}$ solves equation 
\eqref{eq:BSDE_trunc} and the definition of  $({p}^{\infty},{q}^{1,\infty},...,\tilde{q}^{d,\infty})$ the above equality can be rewritten as
 \begin{equation}\begin{split}p^{\infty}_t-p^{\infty}_T &=\int_t^{\lfloor t\rfloor+1}\!\!\!\! \Lambda^*_s p^{\infty}_sds+\sum_{i=1}^d\int_t^{\lfloor t\rfloor+1}\!\!\!\! (\Gamma^{i}_s)^* q^{i,\infty}_sds+\int_t^{\lfloor t\rfloor+1}\!\!\!\!\Psi_sds+\sum_{i=1}^d\int_t^{\lfloor t\rfloor+1}\!\!\!\! q^{i,\infty}_s dW^i_s
 \\
 & \quad + \sum_{n=\lfloor t\rfloor+1}^{\lfloor T\rfloor-1}
 \left[
\int_n^{n+1} \!\!\!\!\Lambda^*_s p^{\infty}_sds+\sum_{i=1}^d\int_n^{n+1}\!\!\!\! (\Gamma^{i}_s)^* q^{i,\infty}_sds+\int_n^{n+1}\!\!\!\!\Psi_sds+\sum_{i=1}^d\int_n^{n+1} \!\!\!\! q^{i,\infty}_s dW^i_s \right]\\
&\quad +\int_{\lfloor T\rfloor}^T\!\! \Lambda^*_s p^{\infty}_sds+\sum_{i=1}^d\int_{\lfloor T\rfloor}^T\!\! (\Gamma^{i}_s)^* q^{i,\infty}_sds+\int_{\lfloor T\rfloor}^T\!\!\Psi_sds+\sum_{i=1}^d\int_{\lfloor T\rfloor}^T\!\! q^{i,\infty}_s dW^i_s\\
&= \int_t^T\!\!\Lambda^*_s p^{\infty}_sds+\sum_{i=1}^d\int_t^T\!\! (\Gamma^{i})^*_s q^{i,\infty}_sds+\int_t^T\!\!\Psi_sds+\sum_{i=1}^d\int_t^T\!\! q^{i,\infty}_s dW^i_s
\end{split} \end{equation}
and this completes the proof of existence of a solution to equation \eqref{eq:BSDE_int}.

$ $

\noindent \textit{Uniqueness:} Let $(p_t,q^1_t,...,q^d_t)_{t\geq 0}$ be a solution to equation \eqref{eq:BSDE_int}.  We choose $\rho\in L^2(\Omega\times [0,\infty[;\mathbb{R}^n)$ with support in the finite interval $[0,T]$ ($\rho_r = 0$, if $r \geq T$) and $\eta \in L^2(\Omega, \mathcal{F}_t,\mathbb{P}; \mathbb{R}^n)$.
Noticing that $(p_t,q^1_t,...,q^d_t)_{t\geq 0}$ is, in particular a solution to equation \eqref{eq:BSDE_trunc} in $[0,T]$ with $\nu=p_T$ by Lemma \ref{l:dualityFO} we get:
$$
\E\int_t^T \braket{\Y_s^{t,\eta,\rho}, \Psi_s}ds + \E\braket{p_T, \Y_T^{t,\eta,\rho}} = \sum_{i=1}^d\E\int_t^T \braket{ q^i_s, \rho_s}ds + \E \braket{\eta, \tilde p_t} .$$
We notice that since $\rho_t=0$ for $t>T$ then by \eqref{stimadiY} we have that $\E\abs{\Y_s^{t,\eta,\rho}}^2\leq Ce^{-2\beta(s-t)}$ for a suitable $C$. So letting $T\rightarrow \infty$ in the above equality we get (recall that $\sup_{t\geq 0}\E|p_t|^2< \infty $ by definition of solution):
\begin{equation}\label{eq:dualityIO}
\E\int_t^{\infty} \braket{\Y_s^{t,\eta,\rho}, \Psi_s}ds= \sum_{i=1}^d\E\int_t^T \braket{ q^i_s, \rho_s}ds + \E \braket{\eta, \tilde p_t}
\end{equation}
and this completes the proof of uniqueness due to the arbitrariness of $t,T,\rho$ and $\eta$. \end{proof}
As a by-product of the above proof we have the following infinite-horizon version of the duality relation:
\begin{corollary}\label{cor:dualityIO}
Let $(p_t,q^1_t,...,q^d_t)_{t\geq 0}$ be a solution to equation \eqref{eq:BSDE_int}.  Fix $\rho\in L^2(\Omega\times [0,\infty[;\mathbb{R}^n)$ with support in  $[0,T]$, $ t\in [0,T)$ and $\eta\in L^2(\Omega, \mathcal{F}_t,\mathbb{P};\mathbb{R}^n)$  then \eqref{eq:dualityIO} holds.

\end{corollary}

\section{Necessary Ergodic SMP}

We give two versions of the SMP in its necessary form. The first is based on the well-posedness result for the infinite horizon BSPDE. The second one is written in terms of the family of truncated backward equations introduced in the previous section. The Hamiltonian associated to the system is 
\begin{equation}
H(x,u,p,q^1,...q^d) = \braket{b(x,u),p} + \sum_{i=1}^d\braket{\sigma^i(x,u),q^i} + f(x,u).
\end{equation}
We are now in a condition to formulate a necessary condition corresponding to the ergodic control problem.
\begin{theorem}[SMP infinite horizon case]\label{thm:necessary inf horizon}
Suppose that $(\bar{X},\bar{u})$ is an optimal pair for the control problem $J^{\inf} $ or $J^{\sup}$ and let $(p^\infty, q^\infty) =  (p^\infty,q^{\infty,1},...,q^{\infty,d})$ be the solution of equation \eqref{eq:BSDE}. Then under Hypothesis 1, the following variational inequality holds:
\begin{equation}
\begin{split}
0&\leq \limsup_{T \to \infty} \frac{1}{T} \E\int_0^T \braket{ D_u H\lp \bar{X}_t, \bar{u}_t, p^\infty_t, q^\infty_t \rp , u_t-\bar{u}_t}_{\R^l} dt,
\end{split}
\end{equation}
where $H(x,u,p,q)$ is the Hamiltonian of the system, and $u(\cdot)$ is an arbitrary admissible control. 
\end{theorem}

\begin{proof}
Let $v(\cdot)= u(\cdot)-\bar{u}(\cdot)$ and let $Y_t$  be the solution to equation \eqref{eq:first:variation}. Lemma \ref{l:dualityFO} with $t=0$, $\eta=0$, $\nu=p^{\infty}_T$, $\gamma= D_u b(\bar{X},\bar{u})$, $\rho^{i}=D_u \sigma(\bar{X},\bar{u})v$ yields
\begin{equation}
\begin{split}
& \E\int_0^T \braket{D_x f(\bar{X}_t,\bar{u}_t),Y_t}dt \\
\quad &= \E\braket{p^\infty_T,Y_T} + \E\int_0^T \braket{p^\infty_t,D_u b(\bar{X}_t,\bar{u}_t)v_t} dt  + \E\int_0^T \braket{q^\infty_t,D_u \sigma(\bar{X}_t,\bar{u}_t)v_t} dt. 
\end{split}
\end{equation}
So that, from Lemma \ref{l:derivative_J} and the relation above, we have
\begin{equation*}
\begin{split}
0& \leq \limsup_{T \to \infty} \frac{1}{T} \E\int_0^T \left[\braket{ D_xf(\bar{X}_t,\bar{u}_t), Y_t}_{\mathbb{R}^n}+ \braket{ D_uf(\bar{X}_t,\bar{u}_t), v_t }_{\mathbb{R}^l}  \right]dt.\\
& \leq -\limsup_{T \to \infty} \frac{1}{T}\E \braket{Y_T,p^\infty_T} + \limsup_{T \to \infty} \frac{1}{T} \E\int_0^T \Big[ \braket{ D_uH(\bar X_t, \bar u_t, p^\infty_t, q^\infty_t),  v_t}_{\R^l} \Big] dt .
\end{split}
\end{equation*}
Recalling that $sup_{t\geq 0}\mathbb{E}|p^{\infty}_t|^2<+\infty$
by definition of solution to equation \eqref{eq:BSDE} and 
 $sup_{t \geq 0}\mathbb{E}|Y_t|^2<+\infty$ by \eqref{stimadiY} we can conclude that
\begin{equation*} 
0 \leq \limsup_{T \to \infty} \frac{1}{T} \E\int_0^T \Big[ \braket{ D_uH(\bar X_t, \bar u_t, p^\infty_t, q^\infty_t),  v_t}_{\R^l} \Big] dt
\end{equation*}
and the claim is proved
\end{proof}

\begin{remark}[SMP truncated case]\label{thm:necessary truncated}
Similarly we can prove a truncated version of the stochastic maximum principle that involves the solution $(p^T,q^{1,T}, ... q^{d,T})$ of equation \eqref{eq:BSDE_trunc} with $\eta=0$.
Indeed if  $(\bar{X},\bar{u})$ is an optimal pair for the control problem 
\eqref{ergodic cost}, then under Hypothesis 1 the following variational inequality holds
\begin{equation}
0\leq\limsup_{T \to \infty} \frac{1}{T} \E\int_0^T \braket{ D_u H\lp \bar{X}_t, \bar{u}_t, p^T_t, q^T_t \rp , u_t-\bar{u}_t}_{\R^l} dt,
\end{equation}
where $H(x,u,p,q)$ is the Hamiltonian of the system and $u(\cdot)$ is an arbitrary admissible control. 
\end{remark}

\begin{proof}
Let $v_t = u_t-\bar{u}_t$, for every $u_t$ admissible. The result easily follows combining Lemma \ref{l:derivative_J} with a duality argument. Precisely, choose $\eta = 0$, $\nu_t = D_u b\lp \bar{X}_t,\bar{u}_t \rp v_t$ and $\Psi_t = D_x f(\bar{X}_t,\bar{u}_t)$ in the general formula \eqref{eq:duality:general}.
\end{proof}

\section{Sufficient SMP}

In this part we prove that under some additional convexity assumption on the Hamiltonian 
function $H$, the variational inequality obtained in Theorem \ref{thm:necessary inf horizon} (the same hold also for Theorem \ref{thm:necessary truncated}) is sufficient for optimality. 

\begin{theorem}[Sufficient SMP]\label{thm:sufficient infinite}
Let $u^*(\cdot) \in \mathcal{U}_{ad}$ be an admissible control, $X^*$ be the corresponding state process and $p^*$ the first adjoint process on infinite time horizon solving \eqref{eq:BSDE} for the couple $(u^*,X^*)$.
Further, let $(x,u) \mapsto H(x,u,p^*_t,q^*_t)$ be a convex function $d \mP \times dt-$a.e. and the following minimality condition holds
\begin{equation}\label{eq:suf minimality}
\limsup_{T \rightarrow +\infty} \frac{1}{T} \E \int^T_0 \braket{ D_u H(X^*_t,u^*_t,p^*_t,q^*_t)
, u_t - u^*_t}_{\R^l} dt \geq 0,
\end{equation}
for every $u(\cdot) \in \mathcal{U}_{ad}$. Then 
$u^*(\cdot)$ is optimal both for $\liminf$ and $\limsup$ formulations of the ergodic control problem.  
\end{theorem}

\begin{proof}
Let $u(\cdot) \in \mathcal{U}_{ad}$ be arbitrary but fixed. Then the goal is to show that
the difference $J(u^*(\cdot)) - J(u(\cdot)) $ is non-positive. Using the sub additivity of the $\limsup$ we have
\begin{align*}
J(u^*(\cdot)) &- J(u(\cdot)) \leq \limsup_{T \rightarrow +\infty} \frac{1}{T} \E \int^T_0 
\left[ f\lp X^*_t, u^*_t \rp - f\lp X_t, u_t \rp \right] dt \nonumber \\
&\quad = \limsup_{T \rightarrow +\infty} \frac{1}{T} \E \int^T_0
\left[ H ( X^*_t, u^*_t, p^{*}_t, q^{*}_t ) - H( X_t, u_t, p^{*}_t, q^{*}_t) \right] dt
\nonumber \\ 
& \quad + \limsup_{T \rightarrow +\infty} \frac{1}{T} \E \int^T_0  
\braket{ b \lp  X_t, u_t \rp - b \lp  X^*_t, u^*_t \rp  , p^{*}_t} dt \\
& \quad + \limsup_{T \rightarrow +\infty} \frac{1}{T} \E \int^T_0  
\braket{ \sigma \lp  X_t, u_t \rp - \sigma \lp  X^*_t, u^*_t \rp  , q^{*}_t} dt = I_1 + I_2 + I_3.  
\end{align*}

\noindent Now, due to convexity of $H$, the term $I_1$ can be estimated from above as follows

\begin{align*}
I_1 & \leq \limsup_{T \rightarrow +\infty} \frac{1}{T} \E \int^T_0
\braket{ D_x H(X^*_t, u^*_t, p^{*}_t, q^{*}_t ), X^*_t - X_t} dt \nonumber \\
& \quad +  \limsup_{T \rightarrow +\infty} \frac{1}{T} \E \int^T_0
\braket{ D_u H(X^*_t, u^*_t, p^{*}_t, q^{*}_t), u^*_t - u_t}_U dt \nonumber \\
& \leq \limsup_{T \rightarrow +\infty} \frac{1}{T} \E \int^T_0
\braket{ D_x H(X^*_t, u^*_t, p^{*}_t, q^{*}_t), X^*_t - X_t } dt, 
\end{align*}
where in the last step we have used the minimality condition \eqref{eq:suf minimality}.
Next,
\begin{equation}
\lim_{T \to \infty} \frac{1}{T} \E \braket{p^*_T,X^*_T - X_T} = 0,
\end{equation}
due to the fact that $p^*,X,X^* \in L^{\infty}\lp \R_+; L^2(\Omega;H) \rp $. 

\noindent By applying the It\^o formula to $\braket{p^{*,T}_T, X^*_T - X_T}_H$ and putting all the terms together we arrive at
\begin{equation}
J(u(\cdot)) - J(u^*(\cdot)) \leq 0.
\end{equation}
The above inequality means that $u^*(\cdot)$ is optimal control.
\end{proof}

 \noindent The form of minimality condition \eqref{eq:suf minimality} is related to our definition of the Hamiltonian. In fact, one could introduce an another sign convention for $H$, namely $H(x,u,p,q) = \braket{b(t,x,u),p} + \sum_{i=1}^d \braket{\sigma^i(x,u),q^i} - f(x,u)$ which would lead to the corresponding modification in the driver of the first adjoint equation, concavity assumption (instead of convexity) on $H$ in $(x,u)$ and the opposite inequality in \eqref{eq:suf minimality}. All these changes would lead to the maximality condition usually considered with stochastic maximum principle. \smallskip

\bibliography{mybib}
\bibliographystyle{plain}

\end{document}